\documentclass[11pt]{article}
\usepackage{latexsym, amsmath, amssymb}

\usepackage{amsmath}
\usepackage{amsfonts}
\usepackage{amssymb}
\usepackage{amscd}
\usepackage{amsthm}
\usepackage{fancyhdr}
\theoremstyle{plain}
\newtheorem{theorem}{Theorem}[section]
\newtheorem{proposition}[theorem]{Proposition}
\newtheorem{remark}[theorem]{Remark}

\newtheorem{definition}[theorem]{Definition}

\newtheorem{example}[theorem]{Example}

\newcommand{\rk}{\mathrm{rank}}

\def\Q{\mathbb Q}

\newcommand\sO{{\mathcal O}}

\def\pee#1{\hbox{$ {\mathbf P}^{#1}$}}
  \def\peen{\hbox{$ {\mathbf  P}^n$}}

  \def \tab#1{\kern #1 truein}

  \def\E{\hbox{${\cal E}$}}
  \def\F{\hbox{${\cal F}$}}

  \def\Q{\hbox{${\cal Q}$}}

  \def\O#1{\hbox{${\cal O}_{#1}$}}  %the structure sheaf of a scheme

\def\mapright#1{\smash{
   \mathop{\longrightarrow}\limits^{#1}}}

\begin{document}
 \title{Cohomological Characterization of
Vector Bundles on Grassmannians of Lines} 
\author{Enrique Arrondo\thanks{Supported in part by the Spanish
Ministry of Education through the research project MTM2006-04785} and
Francesco Malaspina
\vspace{6pt}\\
{\small   Universidad Complutense de Madrid}\\ 
{\small\it Plaza de Ciencias, 3. 
28040 Madrid , Spain}\\ 
{\small\it e-mail: arrondo@mat.ucm.es }\\
\vspace{6pt}\\ 
 {\small  Politecnico di Torino}\\
{\small\it  Corso Duca degli Abruzzi 24, 10129 Torino, Italy}\\
{\small\it e-mail: malaspina@calvino.polito.it}}    \maketitle \def\thefootnote{}
\footnote{\noindent Mathematics Subject Classification 2000: 14F05, 14J60.
\\  Keywords: Universal Bundles on Grassmannians,
Castelnuovo-Mumford regularity.}
  
\begin{abstract} We introduce a notion of regularity for 
coherent sheaves on Grassmannians of lines. We use this notion to
prove some extension of Evans-Griffith criterion to characterize
direct sums of line bundles. We also give a
cohomological characterization of exterior and symmetric powers of
the universal bundles of the Grassmannian.
\end{abstract}  
  
\section*{Introduction}

The notion of Mumford-Castelnuovo regularity of sheaves on the
projective space, introduced in \cite{m}, has shown a very powerful
tool, especially to study vector bundles. This theory allows to
prove easily Horrocks criterion to characterize direct sums of line
bundles as those bundles without intermediate cohomology, and its
improvement by Evans-Griffith depending on the rank of the vector
bundle.
There have been several generalizations of this notion of regularity
to other ambient spaces such as Grassmannians (\cite{c}), products of
projective spaces (\cite{hw}, \cite{cm2}) or quadrics (\cite{BM}). In
most of the cases, the starting point is some variant of the
Beilinson spectral sequence, so that the notion of regularity
consists of a finite number of cohomological vanishings. Such a
notion has a nice behaviour, in particular it can be proved
that, if a coherent sheaf $\F$ is regular, so is any positive twist
of it.

An easy approach to the Mumford-Castelnuovo regularity on the
projective space is through the Koszul exact sequence, obtained from
the Euler exact sequence. In fact, the definition of regularity of
a sheaf can be done by imposing the vanishing of some cohomology of
the terms appearing in the Koszul exact sequence twisted by the
sheaf. In this paper, we explore this approach for Grassmannians of
lines (the right generalization of the Koszul exact sequence becomes
too complicated; this is why we concentrate in these particular
Grassmannians). In order to make the theory to work well, we will
need to impose in the definition the property that any positive
twist of a regular sheaf is still regular. This means that our
notion, consists of infinitely many cohomological vanishings. This
will not be, however, a problem for the applications we have in mind
(and in fact our definition include some vector bundles which are
not regular in the sense of \cite{c}).

We dedicate a first section to recall all the preliminaries we will
need for the Grassmannians of lines, with special attention to the
universal bundles and their cohomological properties. We will also
determine the right generalization to these Grassmannians of the
Koszul exact sequence.

In the second section we introduce our notion of regularity, which 
we will call $G$-regularity. We show that the natural candidates
coming from the universal bundles are $G$-regular (Example
\ref{esempi}) and we prove that this notion satisfies analogue
properties to the Mumford-Castelnuovo regularity (Proposition
\ref{newreg}). We also remark that, when $n=2$, i.e. when the
Grassmannian is a projective plane, $G$-regularity coincides with
Mumford-Castelnuovo regularity (Example \ref{piano}) and that, when
$n=3$, i.e. the Grassmannian is a quadric, the notion
of $G$-regularity containd the stronger notion of regularity given
in \cite{BM}. We finish this section with the first strong
application of our theory: a generalization of Evans-Griffith
criterion to characterize direct sums of line bundles (Theorem
\ref{eg}).

In the last section we prove our main results. We first give two
criteria that, with a finite number of cohomological vanishings,
imply that a vector bundle contains as a direct summand an
exterior power of one universal bundle (Theorem \ref{caratt-wed}) or
a symmetric power of the other universal bundle (Theorem
\ref{caratt-sim}). With the same techniques of those results, we
also give a cohomological characterization of those vector bundles
that are direct sums of twists of the above exterior and symmetric
powers (Theorem \ref{mt}). In particular, for $n=3$ we reobtain for
the four-dimensional quadric the characterizion of the vector
bundles without intermediate cohomology of \cite{Kn}, while for
$n=4$ we reobtain the characterization of direct sums of line
bundles and twists of the universal bundles or their duals given in
\cite{AG}.
     
\section{Preliminaries}
     
Throughout the paper $\peen$ will denote the projective space
consisting of the one-dimensional quotients of the
$(n+1)$-dimensional vector space $V$, while $G(1,n)$ (frequently
denoted just by $G$) will be the Grassmann variety of lines in
$\peen$. We recall the universal exact sequence on
$G=G(1,n)$:
\begin{equation}\label{u}
0\to S^\vee\mapright{\varphi} V\otimes{\cal O}_G\mapright{\psi} Q\to 0 
\end{equation}
defining the universal bundles $S$ and $Q$ over $G$, of respective ranks
$n-1$ and $2$. We will also write
$\sO_G(1)=\bigwedge^2Q\cong\bigwedge^{n-1}S$. In particular, we have
natural isomorphisms
\begin{equation}\label{dual-Q}
S^jQ^\vee\cong(S^jQ)(-j)
\end{equation}
(where $S^j$ denotes the $j$-th symmetric power) and
\begin{equation}\label{dual-S}
\bigwedge^j S^{\vee}\cong\bigwedge^{n-1-j}S(-1)
\end{equation}
Recall that the Pl\"ucker embedding of $G$ is defined by the quotient
$\bigwedge^2V\otimes\sO_G\mapright{\wedge^2\psi}\sO_G(1)$, or equivalently
by the quotient
$\bigwedge^{n-1}V^*\otimes\sO_G\mapright{\wedge^{n-1}\varphi^t}\sO_G(1)$.

The universal sequence \eqref{u} is the analogue in $G$ of the Euler
sequence in the projective space. The long Koszul exact sequence in the
projective space comes by taking the top exterior product in the left map
of the Euler sequence, while the taking smaller exterior products produces
the Koszul exact sequence truncated at the left. In the case of
Grassmannians of lines, for any $j\le n-1$, taking the $j$-th exterior
powers of $\varphi$ in \eqref{u} produces a long exact sequence
\begin{align*}0\to\bigwedge^j
S^{\vee}\to\bigwedge^j V\otimes{\cal
O}_G\to\bigwedge^{j-1}V\otimes
Q\to\dots\to\bigwedge^2V\otimes S^{j-2}Q\to V\otimes
S^{j-1}Q\to S^j Q\to 0.\\
(R_j)
\end{align*}
Dualizing $(R_j)$ and using the canonical isomorphisms \eqref{dual-Q}
we get another exact sequence
\begin{align*}
0\to S^j Q(-j)\to V^*\otimes S^{j-1}Q(-j+1)\to\dots
\to\bigwedge^{j-1}V^*\otimes
Q(-1)\to\bigwedge^j V^*\otimes\O{G}\to\bigwedge^j S\to0\\
(R^\vee_j)
\end{align*}

Observe now that we can glue $(R^\vee_{n-1-j})$ twisted by
$\O{G}(-1)$ with $(R_j)$ and,
\def\cancellato{%%%%%%%%%%%%%%%%%%%%%%%%%
 we obtain
\begin{align*} 0\to S^{n-1-j}Q(-n+j)\to\dots\to\bigwedge^{n-1-j}V^*\otimes{\cal O}_G(-1)\to\bigwedge^j V\otimes{\cal O}_G\to\dots\to S^j Q\to 0\\
(W_j)
\end{align*}
Moreover we can glue $(R_j)$ twisted by $\O{G}(-j)$ with
$(R^\vee_j)$ and we have 
\begin{align*} 0\to\bigwedge^j
S^{\vee}(-j)\to\bigwedge^j V\otimes{\cal
O}_G(-j)\to\bigwedge^{j-1}V\otimes
Q(-j)\to\dots\to\bigwedge^j V^*\otimes\O{G}\to\bigwedge^j S\to0\\
(W'_j)
\end{align*}

In particular, 
}%%%%%%%%%%%%%
when $j=n-1$, we get the analogue of the
Koszul exact sequence:
\begin{align}\label{koszul}
\hskip -4cm0\to\O{G}(-n)\to\bigwedge^{n-1}V\otimes\O{G}(-n+1)\to
\bigwedge^{n-2}V\otimes Q(-n+1)\to\dots\cr
\hskip -0cm\dots\to 
\bigwedge^2V\otimes S^{n-3}Q(-n+1)\to V\otimes S^{n-2}Q(-n+1)\to
V^*\otimes S^{n-2}Q(-n+2)\to\dots
\end{align}
$$\dots\to\bigwedge^{n-2}V^*\otimes
Q(-1)\to\bigwedge^{n-1}V^*\otimes\O{G}\to\O{G}(1)\to0.$$
As we will see, the relevant part of \eqref{koszul} is that the last
morphism is the evaluation morphism for $\O{G}(1)$, and that
\eqref{koszul} defines an element in ${\rm
Ext}^{2n-2}(\O{G}(1),\O{G}(-n))=H^{2n-2}(\O{G}(-n-1))=H^{2n-2}(\omega_G)$,
which is the Serre dual of the unit in $H^0(\O{G})$.

\def\niente{%%%%%%%%%%%%%%%%%%%%%%%%%%%%%%%%%
\begin{align}\label{sym}0\to S^{n-1-j}Q(-n+1+j)\to V^*\otimes
S^{n-2-j}Q(-n+2+j)\to\bigwedge^2 V^*\otimes
S^{n-3-j}Q(-n+2+j)\to\dots&\cr
&\hskip-13cm\dots\to\bigwedge^{n-2-1}V^*\otimes
Q(-1)\to\bigwedge^{n-1-j}V^*\otimes{\cal O}_G\to\bigwedge^j S^{\vee}(1)\to
0.
\end{align}
Finally, glueing together \eqref{wed} and \eqref{sym} obtain a long exact
sequence
\begin{equation}\label{short}0\to
S^{n-1-j}Q(-n+j)\to\dots\to\bigwedge^{n-1-j}V^*\otimes{\cal
O}_G(-1)\to\bigwedge^j V\otimes{\cal O}_G\to\dots\to S^j Q\to
0.\end{equation}
}%%%%%%%%%%%%%%%%%%%%%%%%%%%%%%%%%%%%%

\begin{remark}\label{coomologia}{\rm
We recall that $\bigwedge^jS$ and $S^jQ$ with $0\le j\le n-2$ have
no intermediate cohomology (we say that $E$ on $G$ has no
intermediate cohomology if, for
$i=1,2,\dots,2n-3$ we have the vanishing $H^i_*(E)=0$, i.e.
$H^i(E(k))=0$ for each integer $k$). This is not the case for
$S^jQ$ with
$j\ge n-1$. For example, the exact sequence $(R^\vee_{n-1})$
produces a nonzero element in Ext$^{n-1}({\cal
O}_G(1),S^{n-1}Q(-n+1))=H^{n-1}(S^{n-1}Q(-n))$. In fact this is the
only nonzero intermediate cohomology of $S^{n-1}Q$, while
$(R_j)$ shows that the only nonzero
intermediate cohomology of $S^jQ$ with $j\ge n-1$ is
$H^{n-1}(S^jQ(-n-k))$, with
$k=0,1,\dots,j-n+1$ (observe that, by Serre duality and
\eqref{dual-Q}, it is enough to check the cohomology up to order
$n-1$). We recall that, if $i\le j$, there is a decomposition
\begin{equation}\label{sym-sym} 
S^iQ\otimes S^jQ=
S^{i+j}Q\oplus(S^{i+j-2}Q)(1)\oplus(S^{i+j-4}Q)(2)
\oplus\dots\oplus(S^{j-i}Q)(i)
\end{equation}
so that, again, the only nonzero intermediate cohomology
of any $S^iQ\otimes S^jQ$ is $H^{n-1}(S^iQ\otimes
S^jQ(-n-k))$ for some $k\ge0$. Similarly, using this we deduce
that, for any $i\le n-2$,
$\bigwedge^jS\otimes S^iQ$ has no intermediate cohomology
except for
$$H^{n-1-j}(\bigwedge^jS\otimes S^{n-j-1}Q(-n+j))={\rm
Ext}^{n-1-j}(S^{n-j-1}Q,\bigwedge^jS(-1))$$ (which is generated by
the exact sequence
$(R_{n-1-j})$) and
$$H^{2n-2-j}(\bigwedge^jS\otimes
S^jQ(-n-1))={\rm Ext}^{2n-2-j}(S^jQ(n-j),\bigwedge^jS(-1))$$
(which is generated by the exact sequence obtained by glueing
$(R_{n-1-j})$, $(R^\vee_{n-1-j})$ twisted by $\O{G}(n-1-j)$ and
$(R_j)$ twisted by $\O{G}(n-j)$.

In general, we will call unit element to the extension generating
one of the above cohomological groups.  
 
}\end{remark}

\def\niente{%%%%%%%%%%%%%%%%
In the above sequence appear $(n-j+1)+(j+2)-2=n+1$ bundles. Now we
consider the sequence \eqref{wed} twisted by $(-n+1+j)$ (where  we take
$j=n-1-j$-th exterior powers) followed by the above sequence. We
obtain     

\begin{equation}\label{long}
0\to\bigwedge^{n-1-j} S^{\vee}(-n+j)\to\dots\to V\otimes
S^{n-j-2}Q(-n+j)\to\end{equation}
$$\hskip4cm\to V^*\otimes S^{n-2-j}Q(-n+1+j)\to\dots\to   S^j Q\to 0.$$
In this sequence appear  $(n+1)+(n-j+1)-2=2n-j$ bundles.\\
}%%%%%%%%%%%%%%%%%%

\section{$G$-regularity and Evans-Griffith criterion on $G(1,n)$}

Inspired by \eqref{koszul}, we give the following definition:

\begin{definition}\label{d1} 

{\rm We say that a vector bundle $E$ on $G$ is
{\it $G$-regular} if, for any $k\ge0$, the following conditions hold:
  \begin{enumerate}
  \item[(i)] $H^1(\F\otimes Q(k-1))=H^2(\F\otimes
S^2Q(k-2))=\dots=H^{n-2}(\F\otimes S^{n-2}Q(k-n+2))=0$;
  \item[(ii)] $H^{n-1}(\F\otimes S^{n-2}
Q(k-n+1))=H^n(\F\otimes S^{n-3}
Q(k-n+1))=\dots\\ \\
{}\hspace{4cm}\dots=H^{2n-4}(\F\otimes 
Q(k-n+1))=H^{2n-3}(\F(k-n+1))=0$;
\item[(iii)] $H^{2n-2}(\F(k-n))=0$.
\end{enumerate}
We will say that $\F$ is {\it $m$-$G$-regular} if $\F(m)$ is
$G$-regular. We define the {\it $G$-regularity} of $\F$, $G$-$reg(\F)$,
as the least integer $m$ such that $F(m)$ is $G$-regular. We set $G$-$reg
(\F)=-\infty$ if there is no such an integer. 

}\end{definition}
     
\begin{example}\label{esempi}{\rm 
We get from Remark \ref{coomologia} that the trivial bundle $\O{G}$,
any $\bigwedge^jS$ with $j\in\{1,\dots,n-2\}$ or any $S^jQ$ are
$G$-regular, and in fact their $G$-regularity is zero. This shows
that the definition of Chipalkatti in \cite{c} is much more
restrictive than ours, since $S$ is not regular with his
definition.

\medskip
   
On the other hand, if $T$ is the tangent bundle of the Pl\"ucker
ambient space of $G$, it follows from the restriction of the Euler
exact sequence that $T_{|G}(-1)$ is $G$-regular, while $T_{|G}(-2)$ is
not (because $H^{2n-3}(T_{|G}(-n-1))\ne0$), hence
$G$-$reg(T_{|G})=-1$.

}\end{example}  

We can now prove that our definition of regularity has the right
properties one should expect:

\begin{proposition}\label{newreg} If $\F$ is a $G$-regular coherent
sheaf on $G=G(1,n)$ then, for any $k\geq 0$:
\begin{itemize} 
\item[(i)] $\F(k)$ is $G$-regular.
\item[(ii)] $H^{2n-3}(\F\otimes\bigwedge^jS^\vee(k-n))=0$
for $j=1,\dots,n-2$, and $H^{n-2}(\F\otimes
S^{n-1}Q(k-n+1))=0$.
\item[(iii)] For $j=1,\dots,n-1$, the
multiplication map 
$H^0(\F(k))\otimes H^0(\bigwedge^jS)\to
H^0(\F\otimes \bigwedge^jS(k))$ is surjective.
\item[(iv)] The multiplication map $H^0(\F(k))\otimes
H^0(\O{G}(l))\to H^0(\F(k+l))$ is surjective for any
$l\ge1$.
\item[(v)] $\F(k)$ is generated by its global sections.
\end{itemize}
\end{proposition}

\begin{proof} Part (i) comes from the definition of regularity.
Part (ii) comes by taking cohomology in $(R_j)$
tensored with $\F(n-k)$ for, respectively,
$j=1,\dots,n-1$. Part (iii) follows by taking cohomology
in $(R^\vee_j)$ tensored with $\F$ and having in mind the
identification $H^0(\bigwedge^jS)=\bigwedge^jV^*$.

\medskip

We will prove (iv) by induction on $l$, the case $l=1$
being (iii) for $j=n-1$. The statement for a general $l$
comes from the commutative diagram
$$\begin{matrix}
H^0(\F(k))\otimes H^0(\sO_G(l-1)\otimes H^0(\sO_G(1))
&\to&H^0(\F(k+l-1)\otimes H^0(\sO_G(1))\cr
\downarrow&&\downarrow\cr
H^0(\F(k))\otimes H^0(\sO_G(l))&\to&H^0(\F(k+l))
\end{matrix}$$
using that the top map is surjective by induction hypothesis and the right
map is surjective by applying again (iii) for $j=n-1$.

\medskip

To prove (v), we consider a sufficiently large twist such
that $\F(k+l)$ is generated by its global section.
Consider the commutative diagram
$$\begin{matrix}
H^0(\F(k))\otimes H^0(\sO_G(l))\otimes\sO_G
&\to&H^0(\F(k+l))\otimes\sO_G\cr
\downarrow&&\downarrow\cr
H^0(\F(k))\otimes\sO_G(l)&\to&\F(k+l)
\end{matrix}$$
in which the top map is surjective by (iv) and the right
map is surjective because  $\F(k+l)$ is globally
generated. This yields the surjectivity of
$H^0(\F(k))\otimes\sO_G(l)\to\F(k+l)$, which implies that
$\F(k)$ is generated by its global sections.
\end{proof}

\begin{example}\label{piano}{\rm
If $n=2$, then $G=G(1,2)$ is a projective plane, and $\F$ is $G$-regular
when, for any $k\geq 0$,
$H^1(F(k-1))=H^2(F(k-2))=0$, which coincides with the Castelnuovo-Mumford
regularity on ${\bf {P}}^2$.
}\end{example}

 \begin{example}\label{Q4}{\rm 
If $n=3$ then $G=G(1,3)$ is a quadric hypersurface in $\pee{5}$, where we
have the notion of Qregularity introduced in
\cite{BM}. Specifically, $\F$ is Qregular if
$H^1(\F(-1))=H^2(\F(-2))=H^3(\F(-3))=0$ and $H^4(\F\otimes
Q(-4))=H^4(\F\otimes S(-4))=0$. In particular, $T_{|G}(-1)$ is $G$-regular
but not Qregular (see Example \ref{esempi}), showing that Qregularity
is a stronger condition (in fact, it can be proved that Qregularity implies
$G$-regularity).
\def\niente{%%%%%%%%%%%%%
In fact $F$ is $G$-regular if for any $k\geq 0$
$H^3(F(k-2))=H^4(F(k-3))=0$ and $H^1(F\otimes Q(k-1))=H^2(F\otimes
Q(k-2))=0$.\\
 The two spinor bundles are $\Sigma_1=S$ and $\Sigma_2=Q$. 
 %So if $H^2(F\otimes Q(-2))=H^3(F(-2))=0$, then $H^3(F\otimes S^{\vee}(-2))=H^3(F\otimes S(-3))=0$ and hence $F$ is Qregular (we are using \cite{BM} Proposition $2.4.$).\\
$F$ is Qregular if $H^1(F(-1))=H^2(F(-2))=H^3(F(-3))=0$ and $H^4(F\otimes
Q(-4))=H^4(F\otimes S(-4))=0$. We have $H^4(F\otimes Q(-3))=H^4(F\otimes
S^\vee(-3))=0$ and hence $H^4(F(-3))=0$. $H^4(F\otimes
S^\vee(-3))=H^3(F(-3))=0$ implies $H^3(F\otimes Q(-3))=0$. By a symmetric
argument $H^3(F\otimes S(-3))=0$. $H^3(F\otimes S^\vee(-2))=H^3(F(-2))=0$
implies $H^2(F\otimes Q(-2))=0$. By a symmetric argument $H^2(F\otimes
S(-2))=0$.
  $H^2(F\otimes S)=H^1(F))=0$ implies $H^(F\otimes Q^\vee)=H^1(F\otimes Q(-1))=0$. Then, using the fact that $F(k)$ is Qregular for any $k\geq 0$, we have that $F$ is  $G$-regular.\\
}%%%%%%%%%%%%%%%  
}\end{example}

With our notion of regularity we can prove an analogue of
Evans-Griffith theorem, improving the known results (see \cite{ml} and \cite{o2}) for the total
splitting of vector bundles:

\begin{theorem}\label{eg} A vector bundle $E$ of rank $r$ on
$G=G(1,n)$ splits into a direct sum of line bundles if and only if the
following conditions hold:
\begin{enumerate}
  \item[(i)] $H^1_*(E\otimes Q)=H^2_*(E\otimes S^2Q)=\dots=
H^{n-2}_*(E\otimes S^{n-2}Q)=0$;
  \item[(ii)] $H^{n-1}_*(E\otimes S^{n-2}Q)=H^{n}_*(E\otimes
S^{n-3}Q)=\dots=H^{2n-3-i}_*(E\otimes S^i Q)=0$ with
$i=\big[\frac{2n-2}{r+1}\big]$.
\end{enumerate}
\end{theorem}

\begin{proof} It is clear (see Remark \ref{coomologia}) that a
direct sum of line bundles satisfies (i) and (ii), so that we only
need to prove the converse. The statement is independent of twists by
a line bundle, so that we can assume that $E$ is $G$-regular but
$E(-1)$ is not. In particular, $E$ is globally generated, hence
$E\otimes S^i Q(k+1)$ is (very) ample for any $k\ge0$. This implies,
by Le Potier's vanishing theorem, 
$$H^j(E\otimes S^i Q(k+1)\otimes\O{G}(-n-1))=0$$ 
for $j\ge\rk(E\otimes S^i Q(k+1))$. Hence $H^{2n-3-i}(E\otimes
S^iQ(k-n))=0$ for $i\le\frac{2n-2}{r+1}-1$. This, together with
(i) and (ii), implies that $E(-1)$ satisfies all the conditions of
$G$-regularity except the vanishing of $H^{2n-2}(E(-n-1))$, which is
therefore different from zero. By Serre's duality, we get
$H^0(E^\vee)\ne0$, which together with the fact that $E$ is generated
by its global sections implies that $E$ splits as $E\cong\sO_G\oplus
E'$. The proof is completed by applying the same technique to $E'$,
and making a recursion on the rank.
\end{proof}

\begin{example} {\rm In the particular case $n=3$ our splitting criterion
reads as follows:}\\
Let $E$ be a vector bundle of rank $r$ on the
four-dimensional smooth quadric $Q_4$ such that
$$H^1_*(E\otimes Q)=H^2_*(E\otimes Q)=0$$ 
and, only if $r\ge4$, 
$$H^3_*(E)=0.$$
Then $E$ splits as a direct sum of line bundles.
\end{example}

\section{Characterization of the universal bundles on $G(1,n)$}

After Theorem \ref{eg}, one could have the temptation of
proceeding as in \cite{AG}, i.e. removing from the
statement of the Theorem the conditions not
satisfied by the universal bundles and try see whether
these fewer conditions characterize direct sums of line
bundles and twists of universal bundles. However, this
will not work, since Theorem \ref{eg} already contains
few hypotheses. For example, by Remark \ref{coomologia},
the condition not satisfied by $Q$ is $H^{n-1}_*(Q\otimes
S^{n-2}Q)=0$. However, if we remove that condition, also any $S^jQ$
satisfies the rest of the conditions, so that we cannot hope to
characterize the direct sum of line bundles and twist of $Q$ as
those bundles $E$ satisfying all the hypotheses of Theorem \ref{eg}
except $H^{n-1}_*(E\otimes S^{n-2}Q)=0$. This means that we will need
to add extra conditions to characterize such direct sums.

We will thus first characterize (with just a finite number of
cohomological vanishings) each of the bundles $S^jQ$ or
$\bigwedge^jS$. In a final result, we will put all these results
together to eventually classify direct sums of line bundles,
twists of $Q$ and twists of some $\bigwedge^jS$.

\begin{theorem}\label{caratt-wed}
Let $n\ge3$ and fix $j\in\{1,\dots,n-2\}$. Let $E$ be a vector bundle
on $G=G(1,n)$ such that:
\begin{enumerate}
\item[(i)] $H^{n-1-j}(E\otimes S^{n-1-j}Q(-n+j))\ne0$;
\item[(ii)] $H^1(E(-1))=H^2(E\otimes
Q(-2))=\dots=H^{n-1-j}(E\otimes S^{n-2-j}Q(-n+1+j))=0$;
\item[(iii)] $H^{n-1-j}(E\otimes S^{n-2-j}Q(-n+j))
=H^{n-j}(E\otimes S^{n-3-j}Q(-n+j))
=\dots=H^{2n-3-2j}(E(-n+j))
=0$;
\item[(iv)] $H^{2n-2-2j}(E(-n-1+j))=H^{2n-1-2j}(E\otimes
Q(-n-2+j))=
\dots=H^{2n-3-j}(E\otimes S^{j-1}Q(-n))=0$;
\item[(v)] $H^{2n-2-j}(E\otimes S^{j-1}Q(-n-1))
=H^{2n-1-j}(E\otimes S^{j-2}Q(-n-1))=
\dots=H^{2n-3}(E(-n-1))=0$.
\end{enumerate}
\smallskip
Then $E$ contains $\bigwedge^jS$ as a direct summand. In particular,
a vector bundle $E$ of rank $n-1\choose j$ on $G$ is isomorphic to
$\bigwedge^jS$ if and only if it satisfies (i), (ii), (iii), (iv),
(v).
\end{theorem}

\begin{proof}
By (i), we can take a nonzero element $\alpha\in H^{n-1-j}(E\otimes
S^{n-1-j}Q(-n+j))$. By Serre duality, there exists 
$\beta\in H^{n-1+j}(E^\vee\otimes S^{n-1-j}Q^\vee(-j-1))=
H^{n-1+j}(E^\vee\otimes S^{n-1-j}Q(-n))$, such that the image
of
$\alpha\otimes\beta$ in $H^{2n-2}(\O{G}(-n-1))\cong
H^{2n-2}(S^{n-1-j}Q\otimes S^{n-1-j}Q^\vee(-n-1))$ is the natural
generator (i.e. the dual to the unit of $H^0(\O{G})$). Taking
cohomology in $(R^\vee_{n-1-j})$ and tensorizing with $E(-1)\otimes
H^{n-1+j}(E^\vee\otimes S^{n-1-j}Q^\vee(-j-1))$ and
$S^{n-1-j}Q^\vee(-j)$ we get a commutative diagram:
$$
H^0(E\otimes\bigwedge^jS^\vee)
\otimes H^{n-1+j}(E^\vee\otimes S^{n-1-j}Q(-n))\ \longrightarrow\
H^{n-1+j}(\bigwedge^jS^\vee\otimes S^{n-1-j}Q^\vee(-j-1))$$
$${}\hspace{2cm}\downarrow\sigma\otimes id\hspace{7cm}\downarrow$$
$$H^{n-1-j}(E\otimes S^{n-1-j}Q(-n+j))
\otimes H^{n-1+j}(E^\vee\otimes S^{n-1-j}Q(-n))\ \longrightarrow\
H^{2n-2}(\O{G}(-n-1))
$$
with natural horizontal arrows. We derive from Remark
\ref{coomologia} that the right arrow is an isomorphism of
one-dimensional vector spaces, while condition (ii) implies that
$\sigma$ is an epimorphism. We can thus find $\alpha'\in
H^0(E\otimes\bigwedge^jS^\vee)$ such that $\alpha'\otimes\beta$ maps
to the unit element in $H^{n-1+j}(\bigwedge^jS^\vee\otimes
S^{n-1-j}Q^\vee(-j-1))$.

\bigskip

On the other hand, using Serre duality, the vanishings of (iii) are
equivalent, respectively, to
$$H^{n-1+j}(E^\vee\otimes S^{n-2-j}Q(-n+1))
=H^{n-2+j}(E^\vee\otimes S^{n-3-j}Q(-n+2))
=\dots=H^{2j+1}(E^\vee(-j-1))
=0.$$
In the same way as above, if we consider sequence ($R^\vee_{n-1-j})$ tensored by $E^\vee(-1)$, this shows that $\beta$ lifts to an
element $\beta'\in H^{2j}(E^\vee\otimes\bigwedge^jS^\vee(-j))$ such
that the image of $\alpha'\otimes\beta'$ in
$H^{2j}(\bigwedge^jS^\vee\otimes\bigwedge^jS^\vee(-j))$ is the unit
element. 

\bigskip

Similarly, the vanishings of (iv) are equivalent to
$$H^{2j}(E^\vee(-j))=H^{2j-1}(E^\vee\otimes Q(-j))=
\dots=H^{j+1}(E^\vee\otimes S^{j-1}Q(-j))=0$$
so, if we consider sequence ($R_{j})$ tensored by $E^\vee(-j)$, we see that $\beta'$ can be lifted to $\beta''\in H^j(E^\vee\otimes
S^jQ(-j))$ such that the image of $\alpha'\otimes\beta''$ in
$H^j(\bigwedge^jS^\vee\otimes S^jQ(-j))$ is the unit element.

\bigskip

Finally, the vanishings of (v) are equivalent to 
$$H^j(E^\vee\otimes S^{j-1}Q(-j+1))
=H^{j-1}(E^\vee\otimes S^{j-2}Q(-j+2))=
\dots=H^1(E^\vee)=0$$
which imply that $\beta''$ can be lifted to $\beta'''\in
H^0(E^\vee\otimes\bigwedge^j S)$ such that the image of
$\alpha'\otimes\beta'''$ in
$H^0(\bigwedge^j S^\vee\otimes\bigwedge^j S)$ is the unit element (use sequence ($R^\vee_{j})$ tensored by $E^\vee$. But
this is nothing but saying that, regarding $\alpha'$ as a morphism
$\bigwedge^j S\to E$ and regarding $\beta'''$ as a morphism
$E\to\bigwedge^j S$, their composition is the identity in
$\bigwedge^j S$. In other words, $\bigwedge^j S$ is a direct summand
of $E$, as wanted.
\end{proof}

\begin{theorem}\label{caratt-sim} Let $n\ge3$ and fix
$j\in\{1,\dots,n-2\}$. Let $E$ be a vector bundle on $G=G(1,n)$ such
that:
\begin{itemize}
\item[(i)] $H^{n-1}(E\otimes S^{n-1-j}Q(-n))\ne0$
\item[(ii)] $H^1(E\otimes S^{j-1}Q(-j))=\dots=H^j(E(-j))=0$;
\item[(iii)] $H^{j+1}(E(-j-1))=\dots=H^{n-1}(E\otimes
S^{n-2-j}Q(-n+1) )=0$;
\item[(iv)] $H^{n-1}(E\otimes S^{n-2-j}Q(-n) )= \dots =
H^{2n-3-j}(E(-n))=0$;
\item[(v)] $H^{2n-2-j}(E(-n-1))=\dots =H^{2n-3}(E\otimes S^{j-1}
Q(-n-j))=0$.
\end{itemize}
\smallskip
Then $E$ contains $S^j Q$ as a direct summand.
\end{theorem} 

\begin{proof} We proceed as in the proof of Theorem
\ref{caratt-wed}. By condition (i), we can take a nonzero element
$\alpha\in H^{n-1}(E\otimes S^{n-1-j}Q(-n))$ and its Serre dual
$\beta\in H^{n-1}(E^{\vee}\otimes S^{n-1-j}Q^{\vee}(-1))=
H^{n-1}(E^{\vee}\otimes S^{n-1-j} Q(-n+j))$.

Condition (ii), together with $(R^\vee_{n-1-j})$ tensored by
$E(-j-1)$, implies that we can lift $\alpha$ to
$\alpha'\in H^{n-1-j}(E\otimes\bigwedge^{n-1-j}S(-j-1))=
H^{n-1-j}(E\otimes\bigwedge^jS^\vee(-j))$. Moreover, (iii)
together with $(R_j)$ tensored by $E(-j)$ implies that we
can lift $\alpha'$ to $\alpha''\in H^0(E\otimes S^j
Q(-j))={\rm Hom}(S^j Q,E)$.

On the other hand, writing condition (iv) as
$$H^{n-1}(E^\vee\otimes S^{n-2-j}Q(-n+1-j))= \dots =
H^{j+1}(E^\vee(-1))=0$$
and taking cohomology in $(R^\vee_{n-1-j}$ tensored with
$E^\vee(-1)$, we see that we can lift $\beta$ to $\beta'\in
H^j(E^\vee\otimes\bigwedge^{n-1-j}S(-1))
=H^j(E^\vee\otimes\bigwedge^jS^\vee)$. Writing also condition
(v) as
$$H^j(E^\vee)=\dots =H^1(E^\vee\otimes S^{j-1}Q)=0$$
and taking cohomology in $(R_j)$ tensored with $\E^\vee$ we
get that we can lift $\beta'$ to $\beta''\in
H^0(E^{\vee}\otimes S^jQ)={\rm Hom}(E,S^jQ)$. 

Moreover, $\alpha'',\beta''$ are still dual to each other,
which means that, regarded as morphisms, there composition
is the identity in $S^jQ$. Hence $S^jQ$ is a direct
summand of $E$.
\end{proof}

\begin{theorem}\label{mt} Let $E$ be a vector bundle on
$G=G(1,n)$ with $n\ge3$. Then $E$ is a direct sum of twists of
vector bundles of the form
$\O{G}$, $Q$ or $\bigwedge^jS$ with $j\in\{1,\dots,n-2\}$ if and
only if the following conditions hold:
\begin{enumerate}
\item[(a)] $H^1_*(E)=H^2_*(E\otimes Q)=\dots=H^{n-2}_*(E\otimes
S^{n-3}Q)=0$;
\item[(b)] $H^n_*(E\otimes S^{n-3}
Q)=\dots=H^{2n-4}_*(E\otimes Q)=H^{2n-3}_*(E)=0$;
\item[(c)] for each $j=1,\dots n-2$, \medskip \\
$H^{n-1-j}_*(E\otimes
S^{n-2-j}Q)=H^{n-j}_*(E\otimes S^{n-3-j}Q)
=\dots=H^{2n-3-2j}_*(E)=\medskip \\ 
=H^{2n-2-2j}_*(E)=H^{2n-1-2j}_*(E\otimes
Q)=\dots=H^{2n-3-j}_*(E\otimes S^{j-1}Q)=0$;
\item[(d)] $H^2_*(E)=H^3_*(E\otimes Q)=\dots=H^{n-1}_*(E\otimes
S^{n-3}Q)=\medskip\\ 
=H^n(E\otimes S^{n-4}_*Q)=\dots=H^{2n-4}_*(E)=0$.
\end{enumerate}
\end{theorem}

\begin{proof} It follows from Remark \ref{coomologia} that a direct
sum of twists of  $\O{G}$, $Q$ or $\bigwedge^jS$ satisfies (a), (b),
(c), (d), so that we need to prove the converse. After a twist, we
can assume that $E$ is $G$-regular but $E(-1)$ not. Since $E(-1)$ is
not $G$-regular, and having in mind (b), one of the following
conditions is satisfied:
\begin{enumerate}
\item[(i)] $H^{n-1-j}(E\otimes S^{n-1-j}Q(-n+j) )\ne0$ for some
$j\in\{1,\dots,n-2\}$;
\item[(ii)] $H^{n-1}(E\otimes S^{n-2}Q(-n))\ne0$;
\item[(iii)] $H^{2n-2}(E(-n-1))\ne0$.
\end{enumerate}

In case (i), we are in the hypothesis of Theorem \ref{caratt-wed}
(condition (ii) follows from (a), conditions (iii) and (iv) follow
from (c) and condition (v) follows from (b)). Hence we can write
$E=\bigwedge^jS\oplus E'$ for some other vector bundle $E'$.

\medskip

In case (ii), we are in the hypotheses of Theorem \ref{caratt-sim}
with $j=1$ (condition (ii) follows from (a), conditions (iii) and
(iv) follow from (d) and condition (v) follows from (b)). We can
thus write $E=Q\oplus E'$.

\medskip

Finally, in case (iii) we have, by Serre duality, $H^0(E^\vee)\ne0$.
Since $E$ is generated by its global sections (by Proposition
\ref{newreg}), it follows that we can write $E=\O{G}\oplus E'$.

\medskip

In either case, the new vector bundle $E'$ still satisfies the
hypotheses (a), (b), (c), (d), so that we can conclude by a
recursive argument on the rank.
\end{proof}

\begin{example}{\rm If $n=3$, the hypotheses of Theorem \ref{mt}
reduce to the fact that $E$ has no intermediate cohomology, and we
recover the classification of the ACM bundles on $\Q_4$ proved in
\cite{Kn}.
}\end{example}

\begin{example}{\rm If $n=4$, 
Theorem \ref{mt} characterizes the direct sums of twists of
$\O{G},S,S^\vee$ as those vector bundles $E$ without intermediate
cohomology and such that 
$$H^2_*(E\otimes Q)=H^3_*(E\otimes Q)=H^4_*(E\otimes Q)=0$$
so that we recover \cite{AG} Theorem 2.4.
}\end{example}

\begin{remark}{\rm
It is clear that, for example, in order to characterize direct sums
of line bundles and twists of $Q$, we need to remove condition (c)
in Theorem \ref{mt}, although we will need  more vanishings in (a).
Hence in general, we will need a smaller number of conditions to
characterize more restrictive bundles.

On the other hand, we could have also proceeded as in Theorem
\ref{eg} and use Le Potier vanishing theorem to improve Theorem
\ref{mt} (or any of the variants we just indicated). We preferred
not to do it explicitly, since it represents a small improvement
compared with the difficulty to write it in a clear way.
}\end{remark}


\begin{thebibliography}{99}
  
\bibitem{AG} {\sc E. Arrondo and B. Gra\~{n}a},
\emph{Vector bundles on $G(1,4)$ without intermediate cohomology}, 1999, J. of Algebra 214, 128-142.
  
\bibitem{BM} {\sc E. Ballico and F. Malaspina,} \emph{ Qregularity and an Extension of the
Evans-Griffiths Criterion to Vector Bundles on Quadrics}, J.
Pure Appl Algebra 213 (2009), 194-202.

\bibitem{c} {\sc J. V. Chipalkatti,} \emph{A generalization of Castelnuovo 
regularity to Grassman varieties,} Manuscripta Math.
102 (2000), no. 4, 447--464.

\bibitem{cm1} {\sc L. Costa and R. M. Mir\'{o}-Roig,} \emph{Geometric collections 
and Castelnuovo-Mumford regularity,} arXiv:math/0609561,
Math. Proc. Cambridge Phil. Soc., to appear.

\bibitem{cm2}{\sc L. Costa and R. M. Mir\'{o}-Roig,} \emph{$m$-blocks collections 
and Castelnuovo-Mumford regularity
in multiprojective spaces,} arXiv:math/060956, Nagoya Math. J., to 
appear.

\bibitem%[CM2]
{cm3} {\sc L. Costa and R.M. Mir\'o-Roig,}
\emph{Monads and regularity of vector bundles on projective varieties,} 2007, Preprint. 

\bibitem{e}{\sc E.G. Evans, P. Griffith},
\emph{The syzygy problem}, 1981, Ann. of Math. 114(2), 323-333.

%\bibitem{h} {\sc R. Hern\'andez, I. Sols},
%\emph{On a family of rank 3 bundles on $Gr(1,3)$}, 1985, J. reine
angew Math. 360, 124-135.

\bibitem{hw} {\sc J. W. Hoffman and H. H. Wang,}
\emph{Castelnuovo-Mumford  regularity in biprojective spaces,} Adv.
Geom. 4 (2004), no. 4, 513--536.

\bibitem{Kn} {\sc H. Kn\"orrer,} \emph{Cohen-Macaulay modules on
hypersurface singularities I}, 1987, Invent. Math., 88,
153-164.
\bibitem{ml} {\sc F. Malaspina},
\emph{Few Splitting Criteria for Vector Bundles},  Ricerche mat (2008) 57, 55-64.
\bibitem{m} {\sc D. Mumford,} \emph{Lectures on curves on an algebraic surface, 
Princeton University Press,} Princeton, N.J., 1966.

\bibitem%[Ot2]
{o1}{\sc G. Ottaviani},
\emph{Spinor bundles on Quadrics}, 1988, Trans. Am. Math. Soc., 307,
no 1, 301-316.

\bibitem%[Ot3]
{o2}{\sc G. Ottaviani},
\emph{ Some extension of Horrocks criterion to vector bundles on
Grassmannians and quadrics}, 1989, Annali Mat. Pura Appl. (IV) 155,
317-341.

\end{thebibliography}
\end{document}